\documentclass{article}   
\usepackage{graphicx}
\usepackage[utf8]{inputenc}
\usepackage{amsmath, amssymb, amsfonts, amsthm}
\usepackage{hyperref}
\usepackage{graphicx,url}
\usepackage[shortlabels]{enumitem}
\usepackage{tikz-cd}
\usetikzlibrary{arrows}
\usepackage{indentfirst}
\usepackage{latexsym}
\newcommand{\Ad}{Ad}
\newcommand{\id}{id}

\newcommand{\Hom}{Hom}
\newcommand{\End}{End}
\newcommand{\ad}{ad}
\newcommand{\I}{Im}
\newcommand{\R}{Re}
\usepackage{blindtext}
\newtheorem*{helena}{Proposition}

\newtheorem*{main}{Main Theorem}
\newtheorem{theorem}{Theorem}

\newtheorem{proposition}[theorem]{Proposition}
\newtheorem{lemma}[theorem]{Lemma}
\newtheorem{corollary}[theorem]{Corollary}
\newtheorem{claim}{Claim}

\newtheorem{remark}{Remark}
\newtheorem{definition}{Definition}
\begin{document}
	\title{Holonomy and Equivalence of Analytic Foliations
		\thanks{I am deeply grateful to my thesis advisor Prof. D. Panazzolo who proposed this problem and made crucial contributions in many results in this paper. Without his participation, this paper would not be done. Also, I thank Prof R. Targino for the suggestions to improve the quality of the text. The author is glad to be supported by the Université de Haute-Alsace (UHA), where my Phd training has been well assisted by the IRIMAS laboratory.}}
	
	\author{Francisco Chaves}
	\maketitle
	
	\begin{abstract}
		The main goal of this paper is the analytic classification of the germs of singular foliations generated, up to an analytic change of coordinates, by the germs of vector fields of form the $x\partial_x+\sum_{i=1}^{n}a_i(x,\mathbf{z})\partial_{z_i}$, where $a_i(x,\mathbf{z})$ is a germ of analytic function with $a_i(x,0)=0$. We prove, under some hypothesis, that these germs of singular foliations are analytically classified once their local holonomy along a given separatrix are analytically conjugated.
		
	\end{abstract}
	
	\section{Introduction}\label{intro}
	The analytic classification of singular analytic foliations in dimension two and its connection with the analytic conjugation of the corresponding holonomies was one of the central results of the well-known paper of Mattei and Moussu \cite{ASENS_1980_4_13_4_469_0} in 1980.
	
	Later in 1984, Elizarov and Il'Yashenko \cite{elizarov1984remarks} proved that, in dimension three, if we add some restrictions on the vector field that generates the foliations, the analytic conjugation of the holonomies corresponds to the analytic equivalence of the foliations. In the year 2006, Helena Reis \cite{reis2006equivalence} reproved a result of the same type, but for higher dimensions.
	
	In more details, the authors consider germs of singular analytic vector fields $X$ in $(\mathbb C ^{n},0)$ for $n\geq 3$, with $\lambda_1,\dots,\lambda_n$ as the eigenvalues of the linear part of $X$, verifying:
	\begin{enumerate} 
		\item The origin of $\mathbb C ^{n}$ is an isolated singularity of $X$.
		\item $X$ is of Siegel type (i.e, the convex hull of the eigenvalues of its linear part contains the origin).
		\item All the eigenvalues of the linear part of $X$ are nonzero and there exists a straight line through the origin of $\mathbb C$ separating $\lambda_1$ from the others eigenvalues in the complex plane.
		\item Up to a change of coordinates, $X=\sum_{i=1}^{n}\lambda_iz_i(1+f_i(z))\partial _{z_i}$, where $z=(z_1,\dots,z_n)$, and $f_i$ is a germ of analytic function such that $f_i(0)=0$ for all $i$.
	\end{enumerate}
	In \cite{elizarov1984remarks} and \cite{reis2006equivalence}, it is proved the following.
	
	\begin{helena}[\cite{reis2006equivalence}, Theorem 1]
		Let $X$ and $Y$ be two germ vector fields, verifying $(1), (2), (3)$ and $(4)$. Denote by $h_X$ and $h_Y$ the holonomies of $X$ and $Y$ relatively to the separatrices of $X$ and $Y$ tangent to the eigenspace associated with the first eigenvalue, respectively. Then, if $h_X$ and $h_Y$ are analytically conjugated, $X$ and $Y$ are analytically equivalent.
	\end{helena}
	
	In this paper, we drop the hypothesis $(1)$, and weaken $(2),(3)$, and $(4)$. As consequence, we enlarge the set of vector fields for which the conclusion of the theorem holds.
	
	More precisely, we treat the class of germs of singular analytic foliations called crossing type. A \textit{crossing type}\label{crossing} foliation in $(\mathbb C ^{n+1},0)$ is a triple $(\mathcal F , H, \Gamma)$ such that:
	\begin{enumerate}[label=\roman*.]
		\item $\mathcal F$ is a germ of 1-dimensional analytic foliation.
		\item $H$ is a smooth hyper-surface and $\Gamma$ is a smooth invariant curve such that:
		\begin{enumerate}
			\item $H$ and $\Gamma$ are transverse at the origin.
			\item\label{2.b}Both are invariant by the foliation $\mathcal F$.
		\end{enumerate}
		\item\label{3.b} Each local generator of $\mathcal F$ has a nonzero eigenvalue in the $\Gamma$-direction.
	\end{enumerate}
	\begin{figure}[ht]
		\centering
		\includegraphics[scale=.8]{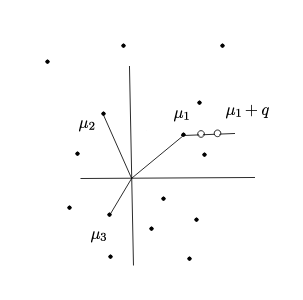}
		\caption{Transversality of $H$ and $\Gamma$  at origin}
	\end{figure}
	
	As in the papers cited above, we have to demand a property on the eigenvalues of the linear part of the local generators. We say that a vector field, with $1,\mu_1\dots,\mu_n$ as the eigenvalues of its linear part, has \textit{no transverse negative resonance}\label{reso} if no element in the positive cone $\mathcal C=\{\sum_{i=1}^{n} p_i\mu_i; p_1+\dots+p_n\geq1\}$, where $p_i\in\mathbb Z_{\geq0}$, can be written in the form $\mu_j+q$, with $q\in \mathbb Z_{\geq1}$, for any $1\leq j\leq n$. 
	\begin{figure}[ht]
		\centering
		\includegraphics[scale=.75]{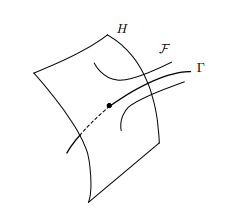} 
		\caption{No transverse negative resonance}
	\end{figure}
	
	As consequence of the definition of crossing type foliation, there exist local coordinates $(x,\mathbf{z})$, so-called adapted to $(\mathcal F,H,\Gamma)$, such that the curve $\Gamma$ and the hypersurface $H$ are expressed respectively by $\Gamma:=\{\mathbf{z}=0\}$, $H:=\{x=~0\}$. Moreover, if the local generators of $(\mathcal F,H,\Gamma)$ have no transverse negative resonance, we can choose a local generator in these adapted coordinates which has the form  \begin{equation}
		x\partial_x+\sum_{i=1}^{n}\sum_{j=1}^{n}a_{ij}z_j\partial_{z_i}+\sum _{i=1}^{n}b_i(x,\mathbf{z})\partial_{z_i},\label{eq2}
	\end{equation}
	where $(a_{ij})_{n\times n}$ is a constant matrix, and $b_i(x,0)=\frac{\partial b_i}{\partial z_j}(x,0)=0$ for all $i,j\in\{1,\dots,n\}$. We say that a vector field of this form is an \textit{$\mathit{x}$-normalized} vector field. 
	
	Our main goal is to classify such singular foliations up to analytic equivalence. Here, we say that two crossing type foliations $(\mathcal F , H,\Gamma)$ and $(\mathcal G,L,\Omega)$ are \textit{analytically equivalent} if there
	exists an analytic change of coordinates mapping the leaves of $\mathcal F$ to the leaves of $\mathcal G$ and the pair $(H,\Gamma)$ to $(L,\Omega)$.
	
	The following theorem is the main result in this paper. We adapted and generalized an idea originally introduced in the thesis of Arame Diaw \cite{diaw2019geometrie,diaw2020pairs} to prove it. Below, we denote by \textit{$\Gamma$-holonomy} and \textit{$\Omega$-holonomy} the respective local holonomies along the curves $\Gamma$ of $\mathcal F$ and $\Omega$ of $\mathcal G$.
	
	\begin{main}\label{main}
		Let $(\mathcal F, H, \Gamma)$ and $(\mathcal G, L,\Omega)$ be two crossing type foliations such that:
		\begin{enumerate}[\fontfamily{pzc}\selectfont (a)]
			\item\label{conda} The linear part of the local generators of $\mathcal F$ and $\mathcal G$ are conjugated.
			\item\label{condb} The local generators of $\mathcal F$ (and therefore that of $\mathcal G $) have no transverse negative resonance.
			\item\label{condc} The respective $\Gamma$-holonomy and $\Omega$-holonomy are analytically conjugated.
		\end{enumerate}
		Then, $(\mathcal F, H, \Gamma)$ and $(\mathcal G, L, \Omega)$ are analytically equivalent.
	\end{main}
	
	Motivated by the Main Theorem, we say that a crossing type foliation $(\mathcal F,H,\Gamma)$ is \textit{analytically classified by its linear part and its $\Gamma$-holonomy} if all crossing type foliations $(\mathcal G,L,\Omega)$ with a conjugated linear part and a conjugated $\Omega$-holonomy to $(\mathcal F,H,\Gamma)$ is analytically equivalent to $(\mathcal F,H,\Gamma)$. 
	
	As a consequence of this theorem, we can give a unified proof for a result obtained with different tools by Mattei and Moussu \cite{ASENS_1980_4_13_4_469_0} and later by Martinet and Ramis \cite{martinet1982problemes}. In a recent paper \cite{diaw2020pairs}, Diaw and Loray also use similar techniques to reprove this theorem. In our notation, we can enunciate it as follows.
	\begin{corollary}\label{mattei}
		
		Consider a crossing type foliation $(\mathcal F,H,\Gamma)\in(\mathbb C^2,0)$ which has, in adapted coordinates, a $x$-normalized local generator of the form $$X=x \partial_x + yf(x,y) \partial_y,$$
		where $f$ is a germ analytic function such that $f(0,0)=\lambda$. Then, two cases can occur.
		\begin{enumerate}
			\item The eigenvalue $\lambda$ belongs to $\mathbb C\setminus\mathbb R_{\leq0}$, then $(\mathcal F,H,\Gamma)$ is analytically linearizable.
			\item The eigenvalue $\lambda$ belongs to $\mathbb R_{\leq0}$, then $(\mathcal F,H,\Gamma)$ is analytically classified by its linear part and its $\Gamma$-holonomy.
		\end{enumerate}
		
	\end{corollary}
	
	We recall that a germ of singular vector field in $(\mathbb C^n,0)$ with eigenvalues $\lambda_1,\dots,\lambda_n$ is in the Siegel (resp. Poincaré) domain if the origin lies (resp. does not lie) in the convex hull of the eigenvalues in $\mathbb C$.
	
	\begin{proof}
		In the Poincaré case ($\lambda\in\mathbb C\setminus\mathbb R_{\leq0}$), the result is immediate since the existence of two analytic separatrices implies that there can be no resonance of Poincaré type. In the Siegel case ($\lambda\in\mathbb R_{\leq0}$), it is sufficient to remark that the condition of no transverse negative resonance is satisfied. Hence, the result is a consequence of the Main Theorem.
	\end{proof}
	
	In dimension three, a generalization of this result, which will be proved in section \ref{sectionfinal}, is the following corollary. 
	
	\begin{corollary}\label{corollary2}
		Consider a crossing type foliation $(\mathcal F,H,\Gamma)\in(\mathbb C^3,0)$ which has, in adapted coordinates, a $x$-normalized local generator with semi-simple part
		
		$$x\partial_x +\lambda y \partial_y+ \mu z\partial_z.$$
		
		Then, three cases can appear: 
		
		\begin{enumerate}[1.]
			\item The eigenvalues are in the Poincaré domain. Then $(\mathcal F,H,\Gamma)$ is analytically normalizable and has at most a finite number of resonant monomials.
			\item The eigenvalues are in the Siegel domain and at least one of the eigenvalues $\lambda,\mu$ is non-real. Then $(\mathcal F,H,\Gamma)$ is analytically classified by its linear part and its $\Gamma$-holonomy.
			\item The eigenvalues are in the Siegel domain and all real. Then, either $(\mathcal F,H,\Gamma)$ is analytically classified by its linear part and its holonomy or one of the following conditions holds up to a permutation of the $y$ and $z$ coordinates:
			\begin{enumerate}
				\item\label{itema} Either  $\mu<\lambda\leq0$ and \begin{equation}\label{corolario2.0}
					p\lambda=\mu+q,
				\end{equation}
				for some $p,q\in\mathbb Z_{\geq1}$.
				\item\label{itemb} Or, $\mu\leq0<\lambda $, and either $(\ref{corolario2.0})$ holds or $\lambda\in\mathbb Q_{>0}-\mu\mathbb Q_{\geq0}$ (notice that these conditions are not mutually exclusive).
			\end{enumerate} 
		\end{enumerate}
	\end{corollary}
	
	The basic tool used in this paper is the concept of $D_{r,R}$-transversely formal series. A \textit{$\mathit{D_{r,R}}$-transversely formal series} is a formal series of the form $\sum_{k_i\in\mathbb N }f_K(x)\mathbf{z} ^K$, where $\mathbf{z} ^K=z^{k_1}\dots z^{k_n}$, and each coefficient $f_K(x)$ is convergent in the annulus $ D_{r,R}:=\{x\in\mathbb C ;r<|x|<R\}$, where $r,R> 0$.
	
	The $D_{r,R}$-transversely formal derivations (derivations over the ring of the $D_{r,R}$-transversely formal series) can be seen as vector fields with coefficients being $D_{r,R}$-transversely formal series. We study the exponential map, the normal form, and the symmetries for such derivations.
	
	After a general study of these objects, we focus on a specific class of vector field. We say that a $D_{r,R}$-transversely formal vector field is \textit{$\mathit{x}$-normalized} if it has the form $(\ref{eq2})$, and $b_1,\dots,b_n$ all lie in the ideal generated by monomials of the form $z_iz_j$ for $i,j\in\{1,\dots,n\}$ in the ring of $D_{r,R}$-transversely formal series.
	
	If the components $f_K$ of a $D_{r,R}$-transversely formal vector field converge in the disk $D_R:=\{z\in\mathbb C ;|z|<R\}$, where $R> 0$, we say that the vector field is a \textit{$\mathit{D_R}$-transversely formal vector field}. Connecting these two classes of vector fields, the following result is a central step to the proof of the Main Theorem.
	
	\begin{theorem}\label{theo1}
		Let $X$ be an $x$-normalized $D_R$-transversely formal vector field that has no transverse negative resonance and $Y$ be an $x$-normalized $D_{r,R}$-transversely formal vector field. If $[X,Y]=0$, then $Y$ is an $x$-normalized $D_R$-transversely formal vector field.
	\end{theorem}
	
	\section{Transversely Formally Objects.}\label{basic}
	
	\subsection{The Ring $\mathcal O_{r,R}[[\mathbf{z}]]$.}
	
	The basic assumptions and notations for this paper are:
	\begin{enumerate}
		\item The set of natural numbers $\mathbb N$ contains zero.
		\item The set $\{e_1,\dots,e_n\}$ is the canonical basis of $\mathbb R^n$.
		\item $\mathcal L_n$ is the set of $n$-uples $K\in\{(\mathbb N^n-e_1)\cup\dots\cup(\mathbb N^n-e_n)\}$, such that $|K|=k_1+\dots+k_n\geq0$.
		\item $\mathcal L_{n,m}$ is the set of $n$-uples $K\in\mathcal L_n$, such that $|K|=k_1+\dots+k_n\geq m$.
		\item The $\sup_{A}f$ denotes the supremum of the function $f$ on the set $A$.
		\item We denote respectively by $\mathcal O_{r,R}$ and $\mathcal O_{R}$ the rings of germs of analytic functions on $D_{r,R}$ and $D_R$.
		\item A point $p\in\mathbb C ^{n+1}$ is denoted by $p=(x,\mathbf{z})=(x,z_1,\dots,z_n)$.
	\end{enumerate}
	
	We now recall the construction of the ring $\mathcal O_{r,R}[[\mathbf{z}]]$ of formal power series of $n$ indeterminates $z_1,\dots,z_n$ with coefficients in $\mathcal O_{r,R}$. For it, we are going to use the notion of inverse limit. 
	
	Consider the ideal $\mathfrak m:=\langle z_1,\dots,z_n\rangle\subset\mathcal O_{r,R}[\mathbf{z}]$. Let $(J^{(i)},(\pi_{ji}))_{j\leq i\in\mathbb N}$ be the inverse system indexed by $\mathbb N$, where $J^{(i)}=\mathcal O_{r,R}[\mathbf{z}]/\mathfrak m^{i+1}$, and for each $j\leq i,$ $\pi_{ji}:J^{(i)}\to J^{(j)}$ is the linear map with kernel $\mathfrak m^{j+1}$ called bonding map. We define $$\mathcal O_{r,R}[[\mathbf{z}]]:=\varprojlim _{j\in \mathbb N}J^{(i)}=\left\{f\in\prod_{i\in\mathbb N}J^{(i)};f_{j}=\pi_{ji}(f_{i}){\text{ for all }}i\leq j\in \mathbb N\right\}.$$
	
	An element $f\in\mathcal O_{r,R}[[\mathbf{z}]]$ is called a \textit{$\mathit{D_{r,R}}$-transversely formal series}, and we can write $f(x,\mathbf{z})=\sum_{|K|=0}^{\infty}f_K(x)\mathbf{z} ^K$, where $\mathbf{z} ^K=z^{k_1}\dots z^{k_n}$, $|K|=k_1+\dots+k_n$, and each coefficient $f_K(x)$ is convergent in the annulus $ D_{r,R}$. We still denote by $\mathfrak m$ the maximal ideal $\langle z_1,\dots,z_n\rangle\subset\mathcal O_{r,R}[[\mathbf{z}]]$.  
	
	\begin{definition}
		We say that $f\in\mathcal O_{r,R}[[\mathbf{z}]]$ is convergent if for each $r',R'$ with $0<r<r'<R'<R$, there exist constants $C,M\in\mathbb R$ such that $\sup_{D_{r',R'}} ||f_K||\leq C M^{|K|}$. We denote by $\mathcal O_{r,R}\{\mathbf{z}\}\subset\mathcal O_{r,R}[[\mathbf{z}]]$ the ring of $D_{r,R}$-transversely convergent series.
	\end{definition}
	
	We  can define exactly in the same way as $\mathcal O_{r,R}[[\mathbf{z}]]$ the subring $\mathcal O_{R}[[\mathbf{z}]]\subset \mathcal O_{r,R}[[\mathbf{z}]]$ of formal power series of $n$ indeterminates with coefficients in $\mathcal O_R$. An element $f\in\mathcal O_{R}[[\mathbf{z}]]$ is called a \textit{$\mathit{D_{R}}$-transversely formal series}.
	
	\begin{definition}\label{1}
		We say that $f \in\mathcal O_R[[\mathbf{z}]]$ is convergent if for each $R'$ with $0<R'<R$, there exist constants $C,M$ such that $\sup_{D_{R'}} ||f_K||\leq C M^K$. We denote by $\mathcal O_{R}\{\mathbf{z}\}\subset\mathcal O_{R}[[\mathbf{z}]]$ the ring of $D_{R}$-transversely convergent series.
	\end{definition}
	
	We observe that diagram of inclusions below follows directly from the definitions.
	
	\begin{figure}[ht]
		\centering
		\begin{tikzpicture}
			% place nodes
			\node at (0, 0)   (a) {$\mathcal{O}_{R}[[\mathbf{z}]]$};
			\node at (2, 0)   (b) {$\mathcal{O}_{r,R}\{\mathbf{z}\}$};
			
			\node at (1, 1)  (c)     {$\mathcal O_{r,R}[[\mathbf{z}]]$};
			\node at (1, -1)  (d)     {$\mathcal{O}_{R}\{\mathbf{z}\}$};
			
			% draw edges
			\draw[right hook->] (a) --  (c);
			\draw[left hook->] (b) -- (c);
			\draw[left hook->] (d) -- (a);
			\draw[right hook->] (d) -- (b);
		\end{tikzpicture}
		\caption{Diagram of inclusions}
	\end{figure}
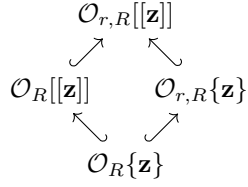
	In the next lemma, we establish that the ring in the last row is the intersection of the two in the middle line.
	
	\begin{lemma}\label{22}
		$\mathcal O_{r,R}\{\mathbf{z}\} \cap \mathcal O_R[[\mathbf{z}]]=\mathcal O_R\{\mathbf{z}\}$.
	\end{lemma}
	
	The proof is based in the next claim. 
	
	\begin{claim}
		Let $f$ be an analytic function on ${D_{R}}$. Then $\sup_{D_{r,R}} ||f||=\sup_{D_{R}} ||f||$, for any $0<r<R$.
	\end{claim}
	\begin{proof}
		The Maximum Modulus Principle for analytic functions guarantees that the $\sup_A ||f||$ is attained on the boundary of $A$. As $\partial D_{r,R}\subset \overline{D_{R}}$ and $\partial D_{R}\subset \partial D_{r,R}$, it follows that $\sup_{D_{r,R}} ||f||=\sup_{D_{R}} ||f||$.
	\end{proof}
	\begin{proof}[of Lemma \ref{22}]
		By the claim above, as $f\in\mathcal O_{r,R}[[\mathbf{z}]]$, each coefficient $f_K$ satisfies $$ \sup_{D_{R'}}||f_K||=\sup_{D_{r',R'}}||f_K||\leq CM^K,$$
		for any $0<r'<R'<R$. Consequently, $f\in\mathcal O_R\{\mathbf{z}\}$.
		The reciprocal is immediate.
	\end{proof}
	
	\subsection{Automorphisms and Derivations.}
	
	A \textit{$\mathit{\mathbb C}$-linear endomorphism} of $\mathcal O_{r,R}[[\mathbf{z}]]$ is defined by a sequence indexed by $\mathbb N$ of $\mathbb C$-linear maps
	$$\Phi^{(i)}\in\Hom_{\mathbb C}(\mathcal O_{r,R}[[\mathbf{z}]],J^{(i)})$$
	such that, for each $j<i$, $\pi_{ji}\circ \Phi^{(i)}=\Phi^{(j)}$, where $\pi_{ji}$ is a bonding map, and
	$$\Phi(f)=\prod_{i\in\mathbb N}\Phi^{(i)}(f).$$
	We denote by $\End_{\mathbb C}(\mathcal O_{r,R}[[\mathbf{z}]])$ the set of such endomorphisms, and we say that $\Phi^{(i)}$ is the $i^{th}$-truncation of $\Phi$. 
	
	A change of coordinates in $\mathcal O _{r,R}[[\mathbf{z}]]$ can be seen as a linear automorphism that preserves $\mathfrak m$. That is, an invertible endomorphism $\Phi\in \End_{\mathbb C}(\mathcal O_{r,R}[[\mathbf{z}]])$, satisfying \begin{enumerate}
		\item $\Phi(f.g)=\Phi(f).\Phi(g)$, $ \forall f,g\in\mathcal O_{r,R}[[\mathbf{z}]]$.
		\item $\Phi(\mathfrak m)\subset\mathfrak m$.
	\end{enumerate}
	We denote by $\mathcal A(\mathcal O _{r,R}[[\mathbf{z}]])$ the group of such automorphisms in $\mathcal O_{r,R}[[\mathbf{z}]]$, and we say that $\Phi\in\mathcal A(\mathcal O _{r,R}[[\mathbf{z}]])$ is a \textit{$\mathit{D_{r,R}}$-transversely formal automorphism}
	
	Similarly, a vector field with components in $\mathcal O_{r,R}[[\mathbf{z}]]$ can be seen as a derivation on this ring. That is, an endomorphism $X\in\End_{\mathbb C}(\mathcal O_{r,R}[[\mathbf{z}]])$, satisfying the Leibniz's rule $$X(f.g)=X(f).g+X(g).f\quad\forall f,g\in\mathcal O_{r,R}[[\mathbf{z}]].$$ We denote by $\mathcal D(\mathcal O_{r,R}[[\mathbf{z}]])$ the $\mathcal O_{r,R}[[\mathbf{z}]]$-module of such derivations that has a structure of Lie Algebra with the Lie Bracket given by
	$$[X,Y]=X\circ Y-Y\circ X.$$
	An element $X\in\mathcal D(\mathcal O_{r,R}[[\mathbf{z}]])$ is called \textit{$\mathit{D_{r,R}}$-transversely formal derivation}. 
	
	\begin{definition}
		A sequence $\{\Phi_{k}\}_{k\in\mathbb N}$ of endomorphism in $\mathcal O_{r,R}[[\mathbf{z}]]$ is called summable if for each $j\in\mathbb N$ there exists a natural number $K=K(j)$ such that the $j^{th}$-truncation of $\Phi_{k}$ is zero for all $k\geq K$.
	\end{definition}
	
	\begin{lemma}\label{summ}
		A summable sequence of endomorphism $\{\Phi_{k}\}_{k\in\mathbb N}$ in $\mathcal O_{r,R}[[\mathbf{z}]]$ defines an endomorphism $\Psi:=\sum_{k=0}^{\infty}\Phi_{k}$.
	\end{lemma}
	\begin{proof}
		By the definition of summable sequence, each $\Psi^{(j)}=\sum_{k=0}^{\infty} \Phi_{k}\mod \mathfrak m ^{j+1}$ is a finite sum of terms, and for $j<i\in\mathbb N$\begin{align*}
			\pi_{ji}\Psi^{(i)}&=\pi_{ji}\sum_{k=0}^{\infty}\Phi_{k}\mod \mathfrak m ^{i+1}\\
			&=\sum_{k=0}^{\infty}\Phi_{k}\mod \mathfrak m ^{j+1}=\Psi^{(j)}
	\end{align*}  \end{proof}
	
	We are particular interested in the case where a sequence of endomorphism is defined by successive powers $X^k=\underbrace{ X\circ \dots\circ X }_k $ of a given derivation $X$, where $X^k(f)$ means applying the derivative $X$ $k$-times on $f\in\mathcal O_{r,R}[[\mathbf{z}]]$.
	
	\begin{definition}
		A vector field $X\in \mathcal D(\mathcal O_{r,R}[[\mathbf{z}]])$ is called \textit{nilpotent} if $X$ preserves the ideal $\mathfrak m$, i.e. $X(\mathfrak m)\subset \mathfrak m$, and for all $j\in\mathbb N$ exists $N=N(j)\in\mathbb N$ such that $$X^{N}(\mathfrak m ^j)\subset \mathfrak m ^{j+1},$$
		where $\mathfrak m^0=\mathcal O_{r,R}[[z]]$. We denote by $\mathcal N(\mathcal O_{r,R}[[\mathbf{z}]])\subset \mathcal D(\mathcal O_{r,R}[[\mathbf{z}]])$ the set of $D_{r,R}$-transversely formal nilpotent vector fields.
	\end{definition}
	
	\begin{proposition}\label{3}
		Let $X\in\mathcal N(\mathcal O_{r,R}[[\mathbf{z}]])$. For any sequence $\{ c_k\}_{k\in\mathbb N}$ of complex numbers, the sequence $\{c_kX^k\}_{k\in\mathbb N}$ is summable.\end{proposition}
	
	\begin{proof}
		
		By the definition, there exist $N(0),N(1),\dots,N(j)\in\mathbb N$, such that
		$$ X^{N(0)}(\mathfrak m^0)\subset\mathfrak m,\quad X^{N(1)}(\mathfrak m)\subset\mathfrak m^2,\dots  X^{N(j)}(\mathfrak m^j)\subset\mathfrak m^{j+1}.$$
		Then, $$ X^{N(0)+\dots+N(j)}(\mathfrak m^0)\subset\mathfrak m^{j+1}$$
		which implies that the sequence $\{c_kX^k\}_{k\in\mathbb N}$ is summable.
	\end{proof}
	\begin{corollary}
		Let $X\in\mathcal N(\mathcal O_{r,R}[[\mathbf{z}]])$. Then, for all $t\in\mathbb C$, the sequence $\left\{\frac{t^k}{k!}X^k\right\}_{k\in\mathbb N}$ defines an endomorphism $\exp(tX)=\sum_{k=0}^{\infty}\frac{t^k}{k!}X^k$ called the time $t$ exponential.
	\end{corollary}
	\begin{proof}
		Making $c_k=\dfrac{t^k}{k!}$, we conclude from Proposition \ref{3} that the series $\sum_{k=0}^{\infty}\frac{t^k}{k!}X^k$ converges.
	\end{proof}
	
	A corollary of the proof of Proposition \ref{3} is the following.
	\begin{corollary}\label{7}
		For all $i\in\mathbb N$, the $i^{th}$-truncation of the automorphism $\exp(tX)$ is polynomial in $t\in \mathbb C$ with coefficients in $\mathcal O_{r,R}$. In other words, for $f\in \mathcal O_{r,R}[[\mathbf{z}]]$, $\exp(tX)(f)=\left(\sum_{k=0}^{\infty}\frac{t^k}{k!}X^k (f)\right)\mod\mathfrak m ^{i+1}$ belongs to $\mathcal O_{r,R}[t][[\mathbf{z}]]$.
	\end{corollary}
	
	In order to study the invertibility of the exponential map, we consider automorphisms and derivations satisfying some flatness conditions. We say that a $D_{r,R}$-transversely formal automorphism $\Phi\in \mathcal A(\mathcal O _{r,R}[[\mathbf{z}]])$ is \textit{tangent to the identity to order} $\mathit{k}$ if $\Phi(x)=x\mod\mathfrak m^{k+1}$ and $\Phi(z_i)=z_i\mod\mathfrak m^{k+1}$ for all $i\in\{1,\dots,n\}$. We denote by $\mathcal A_k(\mathcal O _{r,R}[[\mathbf{z}]])$ the subset of such automorphisms. 
	
	A vector field $X=a(x,\mathbf{z})\partial_x+\sum_{j=1}^nb_j(x,\mathbf{z})\partial_{z_j} \in \mathcal N(\mathcal O_{r,R}[[\mathbf{z}]])$ is \textit{$\mathit{k}$-flat} if $a(x,\mathbf{z})\in \mathfrak m^{k}$, $b_i(x,\mathbf{z})\in \mathfrak m^{k+1}$ for all $i\in\{1,\dots,n\}$. We denote by $\mathcal N_k(\mathcal O_{r,R}[[\mathbf{z}]])$ for the subset of such derivations. 
	
	\begin{proposition}(\cite{serre2009lie}, Theorem 7.2)\label{18}
		For each integer number $k\geq1$, the exponential map $\exp:\mathcal N_k(\mathcal O _{r,R}[[\mathbf{z}]])\to \mathcal A_k(\mathcal O _{r,R}[[\mathbf{z}]])$; $X\mapsto\exp(X)$ is one-to-one with inverse given by, $\log (\Phi)=\sum_{n=0}^{\infty}\frac{(-1)^{n+1}}{n}(\Phi-id)^n$.
	\end{proposition}
	\subsection{The Lie Brackets and the Exponential map.}\label{decom}
	
	In this subsection, we list some results related to the exponential map which are going to be used later. After that, we show the connection between symmetries for a vector field and its centralizer. 
	
	In the next propositions, we consider $\Phi\in\mathcal A(\mathcal O_{r,R}[[\mathbf{z}]])$, $X,Y\in\mathcal N(\mathcal O_{r,R}[[\mathbf{z}]])$, and $Z,W\in\mathcal D(\mathcal O_{r,R}[[\mathbf{z}]])$.
	\begin{proposition}\label{27}
		The following properties hold
		
		\begin{enumerate}
			\item $\Phi^*[Z,W]=[\Phi^*Z,\Phi^*W]$, where $\Phi^*Z=\Phi\circ Z\circ \Phi^{-1}$.
			\item $[Z,fW]=f[Z,W]+Z(f)W$, for $f\in\mathcal O_{r,R}[[\mathbf{z}]]$.
			\item If $[X,Y]=0$, then $\exp (X+Y)=\exp X . \exp Y$.
		\end{enumerate} 
	\end{proposition}
	
	\begin{proof}
		The properties $(1)$ and $(2)$ follow directly from the Lie Bracket's definition, and $(3)$ follows from the Newton's Binomial.
	\end{proof}
	
	\begin{remark}\label{5}
		Since $\exp 0=\id$ and $[X,-X]=0$, the item $(3)$ implies that $(\exp X)^{-1}=\exp -X$ and $\exp (nX)=\exp^n X$, for all $n\in\mathbb Z$.
	\end{remark}
	
	The two maps $\Ad_X,\ad_Z:\mathcal D(\mathcal O_{r,R}[[\mathbf{z}]])\to\mathcal D(\mathcal O_{r,R}[[\mathbf{z}]])$ defined by $$\ad_Z(W)=[Z,W],$$
	$$\Ad_X(W)=\exp X\circ W\circ (\exp X)^{-1},$$ are connected by the following well-known result.
	
	\begin{proposition}\label{9}
		For a fixed $t\in\mathbb C$, $\exp (ad_{tX})Z=\Ad_{tX}Z$.
	\end{proposition}
	
	\begin{proof}
		We give a short proof for the sake of completeness. It is enough to show that $$\ad^n_XZ=\sum_{k=0}^n\binom{n}{k} X^k\circ Z \circ (-X)^{n-k}.$$
		
		We use induction to prove it. The base case $n=1$ is clear, let us assume that it is true for $n=m$. By hypothesis of induction \begin{align*}
			\ad^{m+1}_X Z&=\ad_X\left(\sum_{k=0}^m\binom{m}{k}X^{k}\circ Z\circ (-X)^{m-k}\right)\\
			&=\sum_{k=0}^m\binom{m}{k}X^{k+1}\circ Z\circ (-X)^{m-k}-\sum_{k=0}^m\binom{m}{k}X^{k}\circ Z\circ (-X)^{m-k}X\\
			&=\sum_{k=0}^m\binom{m}{k}X^{k+1}\circ Z\circ (-X)^{m-k}+\sum_{k=0}^m\binom{m}{k}X^{k}\circ Z\circ (-X)^{m-k+1},
		\end{align*}
		By the Pascal's triangle \begin{align*}
			\sum_{k=0}^m\binom{m}{k}X^{k+1}&\circ Z\circ (-X)^{m-k}+\sum_{k=0}^m\binom{m}{k}X^{k}\circ Z\circ (-X)^{m-k+1}=\\
			&=\sum_{k=0}^{m+1}\binom{m+1}{k}X^{k}\circ Z\circ (-X)^{m-k}.
		\end{align*}
	\end{proof}
	
	\begin{definition}
		A $\mathit{D_{r,R}}$-\textit{transversely formal symmetry} for $Z\in\mathcal D(\mathcal O_{r,R}[[\mathbf{z}]])$ is an automorphism $\Phi\in\mathcal A(\mathcal O _{r,R}[[\mathbf{z}]])$ such that $$\Phi\circ Z\circ\Phi ^{-1}=Z.$$
	\end{definition}
	
	We remark that we can similarly define a formal symmetry for a derivation in $\mathcal D(\mathbb C[[x,\mathbf{z}]])$. In addition, Proposition \ref{18} has its version for automorphisms and derivations over $\mathbb C[[x,\mathbf{z}]]$. Then, all the following results of this subsection have their respective statements over the formal objects.
	
	\begin{definition}\label{auto}
		A $D_{r,R}$-transversely formal automorphism $\Phi$ is said to be $\mathit{x}$\textit{-normalized} if it is $\mathcal O_{r,R}$-linear and has the form 
		$$\Phi(x,\mathbf{z})=\left( x,\sum_{i=1}^{n}a_{1i}z_i+\phi_1(x,\mathbf{z}),\dots,\sum_{i=1}^{n}a_{ni}z_i+\phi_n(x,\mathbf{z})\right),$$
		where $(a_{ij})_{n\times n}$ is an invertible constant matrix and $\phi_1,\dots,\phi_n\in\mathfrak m^2\subset\mathcal O_{r,R}[[\mathbf{z}]]$. We denote by $\mathcal{A}_{norm}(\mathcal O_{r,R}[[\mathbf{z}]])$ the set of such automorphisms.
	\end{definition}
	
	Supposing that the maps $\phi_1,\dots,\phi_n$ in Definition \ref{auto} lie in the respective rings $\mathcal O_{R}[[\mathbf{z}]]$ and $\mathbb C[[x,\mathbf{z}]]$, we define similarly $\mathcal{A}_{norm}(\mathcal O_{R}[[\mathbf{z}]])$ and $\mathcal{A}_{norm}(\mathbb C[[x,\mathbf{z}]])$.
	
	In the next result, we use the fact that any $x$-normalized automorphism $\Phi\in\mathcal A_{norm}(\mathcal O_{r,R}[[\mathbf{z}]])$ can be uniquely decomposed as $\Phi=A\circ\Psi$, where $A$ is an invertible linear change of coordinates in the indeterminates $x,z_1,\dots,z_n$, i.e., \begin{align*}
		Ax&=x,\\
		Az_i&=\sum_{j=1}^na_{ij}z_j,\quad a_{ij}\in\mathbb C
	\end{align*} 
	and $\Psi$ is a $D_{r,R}$-transversely formal automorphism tangent to the identity to order 1. Using Proposition \ref{18}, we can write $$\Phi=A\circ\exp Z,$$ for a uniquely determined $D_{r,R}$-transversely formal nilpotent vector field $Z$. We say that this decomposition is the \textit{exponential decomposition} of $\Phi$.
	
	\begin{lemma}\label{19}
		Let $X\in\mathcal D(\mathbb C[[x,\mathbf{z}]])$ be a linear derivation and $\Phi$ an $x$-normalized $D_{r,R}$-transversely automorphism. If $\Phi$ is a $D_{r,R}$-transversely formal symmetry for $X$, and $\Phi=A\circ \exp Z$ is the exponential decomposition of $\Phi$, then the automorphisms $A$ and $\exp Z$ are $D_{r,R}$-transversely formal symmetries for $X$. 
	\end{lemma}
	\begin{proof}
		
		Note that the linear part of $\Phi^*X$ is $A^*X$. As $X$ is linear and $\Phi$ is a symmetry for $X$, we have $A^*X=X$. On the other hand, $X=\Phi^*X=\exp Z^*(A^*X)=\exp Z^*X$, then $\exp Z$ is also a symmetry for $X$.  \end{proof} 
	
	For the following sequence of results, consider $X\in\mathcal N(\mathcal O_{r,R}[[\mathbf{z}]])$ and $Z\in\mathcal D(\mathcal O_{r,R}[[\mathbf{z}]])$.
	
	\begin{lemma}\label{10}
		Let $\exp X$ be a $D_{r,R}$-transversely formal symmetry for $Z$. Then, $\exp tX$ is a $D_{r,R}$-transversely formal symmetry for $Z$ for all $t\in \mathbb C$.
	\end{lemma}
	\begin{proof}
		The following argument is based on the proof of Lemma 2.4.9 \cite{diaw2019geometrie}. By Corollary \ref{7}, we know that for all $i\in\mathbb N$ the $i^{th}$-truncation of $\exp tX$ is polynomial on t, so the map $$P_k(t,x,\mathbf{z})=(\exp tX \circ Z\circ\exp -tX-Z)\mod\mathfrak m ^{k+1}$$ is as well. 
		
		By the Remark \ref{5}, $\exp nX=(\exp X)^n$ for all $n\in \mathbb N$. All these facts together imply that
		\begin{align*}
			\exp nX \circ Z\circ\exp -nX&=\exp ^nX \circ Z\circ(\exp X)^{-n} \\
			&=\exp ^{n-1}X \circ(\exp X\circ Z \circ(\exp X)^{-1})\circ(\exp X)^{-n+1} \\
			&=\exp ^{n-1}X \circ Z\circ(\exp X)^{-n+1} \\
			&\quad \vdots \\&=Z.
		\end{align*}
		
		In other words, for all $n, k\in \mathbb N$, the polynomial $P_k(n,x,\mathbf{z})$ vanishes. As consequence, $P(t,x,\mathbf{z})=0$, therefore, $\exp tX \circ Z\circ\exp -tX=Z$ for all $t\in\mathbb C$. The reciprocal is obvious.
	\end{proof}
	\begin{lemma}\label{11}
		For all $t\in \mathbb C$, the exponential map $\exp tX$ is a $D_{r,R}$-transversely formal symmetry for $Z$ if only if $[X,W]=0$.
	\end{lemma}
	\begin{proof}
		Assume that the $\exp tX$ is a $D_{r,R}$-transversely formal symmetry. The Proposition \ref{9} implies
		$$Z=\Ad_{tX} Z=\exp (\ad_{tX})Z \\
		=(Z+t[X,Z]+\frac{t^2}{2}[X,[X,Z]]+\dots).$$
		Taking the derivative on $t$ and evaluating on zero, we find $[X,Z]=0$. The reciprocal is obvious.
	\end{proof}
	
	\begin{proposition}\label{20}
		The map $\exp X$ is a $D_{r,R}$-transversely formal symmetry for $Z$ if only if $X$ commutes with $Z$, i.e. $[X,Z]=0.$
	\end{proposition}
	
	\begin{proof}
		By Lemma \ref{10}, $\exp X$ is a $D_{r,R}$-transversely formal symmetry for $Z$ if only if for all $t\in\mathbb C$, $\exp tX$ is also a $D_{r,R}$-transversely formal symmetry for $Z$. By Lemma \ref{11}, the $\exp tX$ is a $D_{r,R}$-transversely formal symmetry for $Z$ for all $t\in\mathbb C$ if only if $[X,Z]=0$.
	\end{proof}
	
	\subsection{Normal Form of $x$-normalized $D_R$-transversely Formal Vector Fields.}\label{3.3}
	
	In this subsection, we adapt the Normal Formal Theory of Poincaré (see for instance \cite{SB_1980-1981__23__55_0}) to our present setting of $D_{R}$-transversely formal vector fields. 
	
	We denote by $\mathcal D_{norm}(\mathcal O_{R}[[\mathbf{z}]])$ the set of $x$-normalized derivations in $\mathcal D(\mathcal O_R[[\mathbf{z}]])$. We recall that an element $X\in\mathcal D_{norm}(\mathcal O_{R}[[\mathbf{z}]])$ can be written in the following form $$x\partial_x+\sum_{i=1}^{n}\sum_{j=1}^{n}a_{ij}z_j\partial_{z_i}+\sum _{i=1}^{n}b_i(x,\mathbf{z})\partial_{z_i},$$
	where $a_{ij}\in\mathbb C$ and $b_i(x,\mathbf{z})\in\mathfrak m^2\subset \mathcal O_{R}[[\mathbf{z}]]$.
	
	In the next proof, we use the graded lexicographic order $>_{gr.lex}$ on $\mathbb Z ^n$ defined as follows. For $K=(k_1,\dots,k_n)$ and $L=(l_1,\dots,l_n)$ two $n$-uples in $\mathbb Z^{n}$, we write $K>_{gr.lex}L$ if either $k_1+\dots+k_n=:|K|>|L|:=l_1+\dots+l_n$, or $|K|=|L|$ and $(k_1,\dots,k_n)>_{lex}(l_1,\dots,l_n)$. We say that $K$ is greater than $L$ when $K>_{gr.lex} L$.
	
	In order to simplify the notation, we denote by $L(\lambda)$ the diagonal vector field $\sum_{i=1}^n\lambda_iz_i\partial_{z_i}$, where $\lambda=(\lambda_1,\dots,\lambda_n)\in\mathbb C^n$. 
	\begin{proposition}\label{12}
		For any $X\in\mathcal D_{norm}(\mathcal O_R[[\mathbf{z}]])$ with $1,\mu_{1},\dots,\mu_{n}$ as eigenvalues there exists an $x$-normalized $D_R$-transversely formal change of coordinates which conjugates $X$ to a vector field of the form
		\begin{equation}\label{eq1}
			x\partial_x+L(\mu)+\sum_{i=2}^{n}\epsilon_iz_i^{-1}z_{i-1}L(e_i)+\sum_{K}x^{-\langle\mu,K\rangle}\mathbf{z}^KL(\lambda_{K}),
		\end{equation}
		where $\mu=(\mu_{1},\dots,\mu_{n})$, $\epsilon_i\in\{0,1\}$ is nonzero only if $\mu_i= \mu_{i-1}$, the last sum on the right-hand side is taken over all indices $K\in\mathcal L_{n,1}$ such that $\langle\mu,K\rangle\in\mathbb Z_{\leq 0}$.
	\end{proposition}
	\begin{proof}
		Applying the usual Jordan Normal Theory, we can assume that the linear part of $X$ has the form
		$$X_{lin}=x\partial_x+L(\mu)+\sum_{i=2}^{n}\epsilon_iz_i^{-1}z_{i-1} L(e_i),$$
		where $\mu=(\mu_{1},\dots,\mu_{n})$, $\epsilon_i\in\{0,1\}$ is nonzero only if $\mu_i= \mu_{i-1}$.
		
		We can write the series expansion of $X$ as \begin{equation}
			X_{lin}+\sum_{K,j}g_{K,j}(x)\mathbf{z}^KL(e_j),\label{eq}
		\end{equation}
		where the last sum on the right-hand side is taken over all indices $j\in\{1,\dots,n\}$ and $K\in\mathcal L_{n,1}$, and $g_{K,j}\in\mathcal O_R$. We say that a nonzero term $g_{K,j}(x)\mathbf{z}^KL(e_j)$ in (\ref{eq}) is \textit{resonant} if $\langle \mu,K\rangle$ lies in $\mathbb Z_{\leq0}$ and $g_{K,j}(x)=\lambda_{K,j}x^{-\langle \mu,K\rangle}$, where $\lambda_{K,j}\in\mathbb C$.
		
		We eliminate the nonresonant terms by successive applications of automorphisms of the form 
		$$\Phi=\exp(f(x)\mathbf{z}^KL(e_j)),$$
		where $f\in\mathcal O_R$, $K\in\mathcal L_{n,1}$, and $j\in\{1,\dots,n\}$ are conveniently chosen.
		
		Indeed, we consider the smallest nonresonant term $g_{K_0,j_0}(x)\mathbf{z}^{K_0}L(e_{j_0})$ in (\ref{eq}) with respect to the pair $(K,j)$ and to the gr.lex order. The action of $\Phi=\exp{(f(x)\mathbf{z}^{K_0}L(e_{j_0}))}$ by conjugation on $X$ gives the expression
		\begin{align*}
			\Phi^*X=&X_{lin}+\sum_{(K,j)\leq (K_0,j_0)}\lambda_{K,j}x^{-\langle \mu,K\rangle}\mathbf{z}^KL(e_j)+\\&+(g_{K_0,j_0}(x)-(x\partial_x-\langle\mu,K_0\rangle)f(x))\mathbf{z}^{K_0}L(e_{j_0})+\mathcal{R},
		\end{align*}
		where the rest term $\mathcal R$ is a sum of vector fields $g_{K,j}\mathbf{z}^KL(e_j)$ with $(K,j)>_{gr.lex}(K_0,j_0)$.
		
		Writing the series expansion $g_{K_0,j_0}(x)=\sum_{i\neq\langle\mu,K_0\rangle}a_ix^i+a_{\langle\mu,K_0\rangle}x^{\langle\mu,K_0\rangle}$, with the convention that $a_{\langle\mu,K_0\rangle}=0$ if $\langle\mu,K_0\rangle\in\mathbb Z_{\leq0}$, we can define $f(x)=\sum_{i\neq\langle\mu,K_0\rangle}\frac{a_i}{i+\langle\mu,K_0\rangle}x^i$. As a result, we obtain a new vector field whose smallest nonresonant term is of order strictly greater than $(K_0,j_0)$.
	\end{proof}
	
	We say that the expression (\ref{eq1}) is a \textit{normal form} for $X$. 
	
	\begin{proposition}\label{26}
		Let $X, Y\in\mathcal D_{norm}(\mathcal O_{R}[[\mathbf{z}]])$ and $\Psi\in\mathcal A_{norm}(\mathbb C[[x,\mathbf{z}]])$. If $\Psi$ conjugates $X$ to $Y$, then $\Psi\in\mathcal{A}_{norm}(\mathcal O_{R}[[\mathbf{z}]])$.
	\end{proposition}
	
	\begin{proof}
		By Proposition \ref{12}, there exist $\Phi_1,\Phi_2\in\mathcal{A}_{norm}(\mathcal O_R[[\mathbf{z}]])$ which diagonalize the respective semisimple parts of $X$ and $Y$, in other words, we have the respective Jordan decompositions $$\Phi_1^*X=X_s+X_n=x\partial_x+L(\mu)+X_n,$$ $$\Phi_2^*Y=Y_s+Y_n=x\partial_x+L(\lambda)+Y_n.$$
		By hypothesis, there exists $\Psi\in\mathcal{A}_{norm}(\mathbb C[[x,\mathbf{z}]])$ such that $$\Psi^*X=Y,$$
		which implies $\mu=\lambda$. By the uniqueness of the Jordan decomposition, the automorphism $\Psi_0=\Phi_2\circ\Psi\circ\Phi_1^{-1}\in \mathcal{A}_{norm}(\mathbb C[[x,\mathbf{z}]])$ is such that $$\Psi_0^*(x\partial_x+L(\mu))=x\partial_x+L(\mu).$$
		It means that $\Psi_0$ is an $x$-normalized formal symmetry for $x\partial_x+L(\mu)$.
		
		Consider the exponential decomposition $A\circ \exp Z$ of $\Psi_0$. Since $\Psi_0$ is $x$-normalized, $Z$ has the form $\sum_{K\in\mathcal L_{n,1}} b_K(x)\mathbf{z}^KL(\lambda_K)$, where each $b_K\in\mathbb C[[x]]$. By Lemma \ref{19}, we know that $\exp Z$ is a symmetry for $x\partial_x+L(\mu)$. And by Proposition \ref{20}, we have that $$[Z,x\partial_x+L(\mu)]=0.$$ 
		Using the above expansion of $Z$, this equality is equivalent to state that $$(x\partial_x+\langle\mu,K\rangle )b_K(x)=0$$
		for all $K\in\mathcal L_{n,1}$. Writing $b_K(x)=\sum_{i=0}^{\infty}c_ix^i$, with $c_i\in\mathbb C$, we obtain for each $i\in\mathbb N$
		$$(i+\langle\mu,K\rangle )c_i=0.$$
		
		This implies that either $\langle\mu,K\rangle\notin\mathbb Z_{<0}$ and $b_K=0$ or else $b_K(x)=cx^{-\langle\mu,K\rangle}$ for some
		constant $c\in\mathbb C$. In other words, $Z$ is a $D_{R}$-transversely formal vector field, and consequently, $\exp Z$ is a $D_{R}$-transversely formal change of coordinates. In conclusion, $\Phi=\Phi_2^{-1}\circ\Psi_0\circ\Phi_1$ is a $D_R$-transversely formal automorphism.
	\end{proof}
	
	The property of no transverse negative resonance stated in the introduction
	can be reformulated as follows.
	\begin{definition}
		We say that a vector field $X\in\mathcal D(\mathcal O_{R}[[\mathbf{z}]])$ with $1,\mu_1,\dots,\mu_n$ as eigenvalues has no transverse negative resonance if $$\langle\mu,\mathcal L_{n,0}\rangle\cap\mathbb Z_{\geq 1}=\emptyset,$$
		where $\mu=(\mu_1,\dots,\mu_n)$.
	\end{definition}
	Now, we have all the tools to prove the $Theorem$ \ref{theo1}. Recall that it characterizes the center of an $x$-normalized $D_R$-transversely formal vector field with no transverse negative resonance on the set of $x$-normalized $D_{r,R}$-transversely formal vector fields.
	\begin{proof}[of Theorem \ref{theo1}]
		By Proposition \ref{12}, we can assume that $X$ has the form 
		$$X=x\partial_x+L(\mu)+\sum_{i=2}^n\epsilon_iz_i^{-1}z_{i+1} L(e_i)+\sum_{K}b_Kx^{-\langle\mu,K\rangle}\mathbf{z}^KL(\lambda_{K}).$$
		Since $Y$ is $x$-normalized, it can be expanded as $$Y=x\partial_x+\sum_{i=1}^{n}\sum_{j=1}^{n}c_{ij}z_j\partial_{z_i}+\sum_{l,K}d_{lK}x^l\mathbf{z}^KL(\lambda_{K}),$$
		where $l\in\mathbb Z$, $K\in\mathcal L_{n,1}$.
		
		As $Y$ commutes with $X$, it has to commute with its semisimple part $x\partial_x+L(\mu)$. This implies that for each $l\in\mathbb Z$, $K\in\mathcal L_{n,1}$, we must have 
		$$(l+\langle \mu,K\rangle)d_{lK}x^l\mathbf{z}^KL(\lambda_K)=[x\partial_x+L(\mu), d_{lK}x^l\mathbf{z}^KL(\lambda_{K})]=0.$$
		
		Since $X$ has no transverse negative resonance, both expressions can vanish only if $l\geq0$. This means all the monomials on the expansion of $Y$ have positive exponents on $x$. Consequently $Y\in\mathcal D_{norm}(\mathcal O_R[[\mathbf{z}]])$.
	\end{proof}
	
	\section{The Crossing Type Foliations}\label{sectionfinal}
	The first goal of this section is to define the holonomy for a crossing type foliation. Then, we show that it is always possible to construct an $x$-normalized automorphism that conjugates local generators of crossing type foliations on a neighborhood of $D_{r,R}$ if the conditions \ref{conda}, \ref{condb}, and \ref{condc} (see section \ref{intro}) are satisfied. Using these results, we prove that is always possible to extend this automorphism to a neighborhood of the origin. In other words, we prove the main result of this paper. After that, we show an application in dimension 3. 
	
	We will need the following well-known result.
	
	\begin{lemma}\label{lemma normal form}
		Let $X\in\mathcal D(\mathcal O_{R}\{\textbf{z}\})$ be a vector field having the form \begin{equation}\label{equation crossing type}
			X=x\partial_x+\sum_{i=1}^{n}\sum_{j=1}^{n}a_{ij}z_j\partial_{z_i}+\sum_{i=1}^{n}\sum_{j=1}^{n}b_{ij}(x)z_i\partial_{z_j},
		\end{equation}
		where $a_{ij}\in\mathbb C$, and $b_{ij}\in\mathcal O_R$ vanishes at $0\in\mathbb C$. If $X$ has no transverse negative resonance, then, possibly reducing the radius $R$, there exists an $x$-normalized change of coordinates which conjugates $X$ to a vector field with the form $$x\partial_x+\sum_{i=1}^{n}\sum_{j=1}^{n}a_{ij}z_j\partial_{z_i}.$$
	\end{lemma}
	\begin{proof}
		Let $1,\mu_1,\dots,\mu_n$ be the eigenvalues of $X$. The hypothesis of no transverse negative resonance implies that $\mu_i-\mu_j\notin\mathbb Z_{\geq1}$, for any $i\neq j$. Hence, the expansion (\ref{equation crossing type}) has no resonant term in the sense of Poincaré. Now, note that the vector field $X$ is associated to the following differential system $$\begin{cases}
			\Dot{x}&=x,\\
			\Dot{\textbf{z}}&=(A_0+xA_1+\dots)\textbf{z},
		\end{cases}$$
		where $A_0=(a_{ij})_{n\times n}$, and $A_i$ is a n-dimensional square complex matrix for all $i\in\mathbb Z_{>0}$. Then, by Theorems 16.15 and 16.16 in \cite{ilyashenko2008lectures}, there exists an x-normalized change of coordinates which conjugates $X$ to the vector field $x\partial_x+A_0\textbf{z}$.
	\end{proof}
	
	\begin{lemma}\label{16}
		Let $(\mathcal F , H, \Gamma)$ be a crossing type foliation which its local generators have no transverse negative resonance. Then, there exist local coordinates
		$$(x,\mathbf{z}) = (x, z_1,\dots,z_n)$$
		such that $\Gamma := \{\mathbf{z}=0\}$, $H :=\{x=0\}$. Moreover, once these coordinates are fixed, there exists a unique local generator $X$ for $\mathcal F$ that is an $x$-normalized $D_R$-transversely convergent vector field. 
	\end{lemma}
	\begin{proof}
		We can choose coordinates $(x, \mathbf{z})$ such that $\Gamma := \{\mathbf{z}=0\}$, $H :=\{x=0\}$. The condition (\ref{2.b}) in the definition of crossing type foliation (see section \ref{intro}) implies that a local generator for $\mathcal F$ can be written in these coordinates
		as 
		$$Y_0=g(x,\mathbf{z})x\partial_x+\sum_{i=1}^nh_i(x,\mathbf{z})\partial_{z_i}$$
		where $g,h_1,\dots$, and $h_n$ do not have common factors and $h_1, \dots, h_n$ lie in the ideal $\mathfrak m$. The condition (\ref{3.b}) says that the linear part of $Y_0$ has a nonzero eigenvalue in the $\Gamma$-direction. Then $g(0)\neq0$, and $Y_1=g^ {-1} Y$ is a local generator of $\mathcal F$.
		
		By hypothesis, $Y_1$ has no transverse negative resonance, then, by Lemma \ref{lemma normal form}, there exists an $x$-normalized $D_R$-transversely convergent change of coordinates that conjugates $Y_1$ to $$X=x\partial_x+\sum_{i=1}^{n}\sum_{j=1}^{n}a_{ij}z_j\partial_{z_i}+\sum _{i=1}^{n}b_i(x,\mathbf{z})\partial_{z_i},$$
		where $(a_{ij})_{n\times n}$ is a constant matrix, and $b_i\in\mathfrak m^2\subset\mathcal O_R\{\mathbf{z}\}$ for $1\leq i\leq n$.
	\end{proof}
	
	We say that the coordinates given by this lemma are \textit{adapted} to the crossing type foliation $(\mathcal F , H, \Gamma)$ and that $X$ is the \textit{$\mathit{x}$-normalized local generator} for $\mathcal F$ in these adapted coordinates. 
	
	For any $R>0$, up to a constant rescaling in the $x$ variable, we can suppose that the $x$-normalized local generators are defined in a domain containing a neighborhood of $\overline{D_R\times \{\mathbf{0}\}}\subset\mathbb C^{n+1}$. Then, we fix once for all an constant $R>1$.
	
	In these adapted coordinates, the local $\Gamma$-holonomy can be computed by lifting the circular path $\{(e^{2\pi\theta},0);\theta \in[0,1]\}$ along the leaves of the foliation that goes through a small poli-disk $\mathbb D\subset\mathbb C^{n}$ transverse to $\Gamma^*$ at the point $(1,\mathbf{0})$. 
	
	\begin{definition}
		We say that $(\mathcal F , H, \Gamma)$ and $(\mathcal G , L, \Omega)$ are $\mathit{D_{r,R}}$\textit{-transversely equivalent} if there exist $0< r<1<R$, respective adapted coordinates $(x,\mathbf{z})$ and $(y,\mathbf{w})$, and a bianalytic map $\Psi$ between two open neighborhoods $U,V\subset\mathbb C^{n+1}$ of $D_{r,R}\times \{\mathbf{0}\}$ such that $\Psi$ conjugates the $x$-normalized and $y$-normalized local generators restricted to $U$ and $V$, and we can write $\Psi$ in the form $$(y,\mathbf{w})=\left( x,\sum_{i=1}^{n}a_{1i}z_i+\psi_1(x,\mathbf{z}),\dots,\sum_{i=1}^{n}a_{ni}z_i+\psi_n(x,\mathbf{z})\right),$$
		where $(a_{ij})_{n\times n}$ is an invertible constant matrix and $\psi_1,\dots,\psi_n\in\mathfrak m^2\subset\mathcal O_{r,R}[[\mathbf{z}]]$.
		
	\end{definition}
	We observe that $\Psi$ is an $x$-normalized $D_{r,R}$-transversely automorphism (see Definition \ref{auto}).
	
	Even though the next result is a simple application of Theorem 2 chapter IV \cite{camacho2013geometric}, we are going to give a prove of the new fact that the map which conjugates the local generators is $x$-normalized. We refer to \cite{camacho1977introduccao} for further details.
	\begin{proposition}\label{31}
		Let $(\mathcal F,H,\Gamma)$ and $(\mathcal G,L,\Omega)$ be two crossing type foliations with respective local generators $X$ and $Y$ verifying \ref{conda}, \ref{condb}, and \ref{condc} (see section \ref{intro}). Then, the crossing type foliations are $D_{r,R}$-transversely equivalent.
	\end{proposition}
	
	\begin{proof}
		Let $K\subset\Gamma^*$ be a compact set containing the unity circle $\mathbb S^1$. Given a point $(x,\mathbf{z})$ in a convenient neighborhood of $K$, let $\alpha_{x} :[0,-ln|x|]\to K$ and $\beta_{x}:[0,\Tilde{t}]\to K$ be the curves such that $\alpha_{x}(t)=(xe^t,\mathbf{0})$, $\beta_x(t)=\left(x/|x|e^{-2\pi it},\mathbf{0}\right)$, and $\beta_x(\Tilde{t})=(1,\mathbf{0})$. 
		
		\begin{figure}[ht]
			\centering
			\includegraphics[scale=0.75]{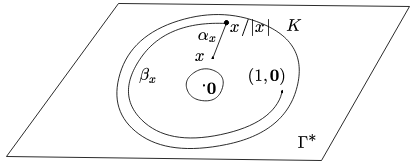} \caption{Composed path}
		\end{figure}
		
		Denote by $f^{\gamma}_{\mathcal F}$ and $f^{\gamma}_{\mathcal G}$ the respective lifts of a curve $\gamma\subset K$ to the leaves of $(\mathcal F,H,\Gamma)$ and $(\mathcal G,L,\Omega)$, and by $\varphi$ the map that conjugates the $\Gamma$-holonomy to the $\Omega$-holonomy. From Theorem~1.2~\cite{camacho1977introduccao}, we recall that the map that conjugates the local generators $X$ and $Y$ is given by $$\Psi(x,\mathbf{z}):=f^{\alpha^{-1}_x}_{\mathcal G}\circ f^{\beta^{-1}_x}_{\mathcal G}\circ \varphi\circ f^{\beta_{x}}_{\mathcal F}\circ f^{\alpha_x}_{\mathcal F}(x,\mathbf{z}).$$
		
		By the hypothesis \ref{conda}, we can assume that $X\mod \mathfrak m^2=Y\mod \mathfrak m^2=x\partial_x+\sum_{i=1}^{n}\sum_{j=1}^{n}a_{ij}z_j\partial_{z_i}$, where $A=(a_{ij})_{n\times n}$ is a constant matrix. As consequence, the holonomies have the same linear part. Hence, the map $\varphi$ has the form $(x,\mathbf{z})\mod \mathfrak m^2$, and as $\Gamma:=\{\mathbf{z}=0\}$ and $\Omega:=\{\mathbf{w}=0\}$, we have $\varphi|_{\{\mathbf{z}=0\}}=\id$.
		
		The restrictions of $X$ to the curves $\alpha_x$ and $\beta_{x}$ are equivalent to the equations below, where $\gamma_1$ and $\gamma_2$ are the respective solutions
		$$\begin{cases}
			\frac{\partial \mathbf{z}}{\partial t}=A\mathbf{z}\mod \mathfrak m^2\\\mathbf{z}(0)=\mathbf{z}\\\gamma_1(t)=\exp({tA})\mathbf{z}\mod \mathfrak m^2
		\end{cases},\begin{cases}
			\frac{\partial \mathbf{z}}{\partial t}=-2\pi iA\mathbf{z}\mod \mathfrak m^2\\\mathbf{z}(0)=\mathbf{z}\\\gamma_2(t)=\exp({-2\pi itA})\mathbf{z}\mod\mathfrak m^2
		\end{cases}.$$
		Then, there exists a constant $\sigma\in\mathbb C$ such that, the compositions of the lifts are
		\begin{align*}
			f^{\beta_x}_{\mathcal F}\circ f^{\alpha_x}_{\mathcal F}(x,\mathbf{z})&=(1,\exp({\sigma A})\mathbf{z})\mod \mathfrak m^2, \\
			f^{\alpha^{-1}_x}_{\mathcal G}\circ f^{\beta^{-1}_{x}}_{\mathcal G}(1,\mathbf{z})&=(x,\exp({-\sigma A})\mathbf{z})\mod \mathfrak m^2.
		\end{align*}
		We conclude that $\Psi(x,\mathbf{z})=(x,\mathbf{z})\mod\mathfrak m^2$ and $\Psi(x,0)=(x,0)$, in other words, the map $\Psi$ is an $x$-normalized automorphism.
	\end{proof}
	\begin{remark}
		Looking at the structure of the logarithm of invertible linear maps, one sees that the condition \ref{conda} of the Main Theorem can be replaced by the following non-equivalent condition: Writing the respective semi-simple parts of $X$ and $Y$ as $x\partial_x+L(\mu)$ and $x\partial_x+L(\lambda)$, no difference $\mu_i-\lambda_j$, $1\leq i,j\leq n$ is a nonzero integer.
	\end{remark}
	We can reformulate the Main Theorem as follows: \textit{If the eigenvalues of $(\mathcal F,H,\Gamma)$ satisfy the no transverse negative resonance condition, then $(\mathcal F,H,\Gamma)$ is analytically classified by its linear part and its $\Gamma$-holonomy}. Now, we give a proof of it. 
	\begin{proof}[of the Main Theorem]
		Let $\Psi\in\mathcal A_{norm}(\mathcal O_{r,R}\{\mathbf{z}\})$ be the automorphism defined by Proposition \ref{31} which conjugates $X$ to $Y$, the respective local generators, in a neighborhood of an annulus $D_{r,R}$.
		
		By Proposition \ref{12}, there exist $\Phi_1,\Phi_2\in\mathcal{A}_{norm}(\mathcal O_R[[\mathbf{z}]])$ which diagonalize the respective semisimple parts of $X$ and $Y$, in other words $$\Phi_1^*X=X_s+X_n,$$ $$\Phi_2^*Y=Y_s+Y_n,$$
		where $X_s=Y_s=x\partial_x+L(\mu)$. By the uniqueness of the Jordan decomposition, the automorphism $\Psi_0=\Phi_2\circ\Psi\circ\Phi_1^{-1}\in \mathcal{A}_{norm}(\mathcal O_{r,R}[[\mathbf{z}]])$ is such that $$\Psi_0^*(x\partial_x+L(\mu))=x\partial_x+L(\mu).$$
		It means that $\Psi_0$ is an $x$-normalized $D_{r,R}$-transversely formal symmetry for $x\partial_x+L(\mu)$.
		
		Let $A\circ \exp{W}$ be the exponential decomposition of $\Psi_0$ (see subsection \ref{decom}). Applying Lemma \ref{19}, the $\exp{W}$ is a symmetry for $x\partial_x+L(\mu)$, and by Proposition \ref{20}, we must have $$[W,x\partial_x+L(\mu)]=0.$$
		Since $x\partial_x+L(\mu)$ has no transverse negative resonance, we can apply Theorem \ref{theo1} to guarantee that $W\in \mathcal D(\mathcal O_{R}[[\mathbf{z}]])$, and then, $\exp W\in \mathcal A(\mathcal O_{R}[[\mathbf{z}]])$.
		
		Finally, we have that $\Psi_0=A\circ \exp W\in\mathcal A( \mathcal O_R[[\mathbf{z}]])$. As consequence, the automorphism $\Psi$ lies in the intersection $\mathcal A( \mathcal O_R[[\mathbf{z}]])\cap \mathcal A( \mathcal O_{r,R}\{\mathbf{z}\})$. Applying Lemma \ref{22} to the components $\Psi_1,\dots,\Psi_{n+1}$ of $\Psi$, we conclude that they lie in $\mathcal O_R\{\mathbf{z}\}$. Therefore, $\Psi\in \mathcal A_{norm}( \mathcal O_R\{\mathbf{z}\})$. \end{proof}
	
	In dimension three, we can be more precise about the local classification in terms of the linear part and the holonomy proving Corollary \ref{corollary2}.
	
	\begin{proof}[of Corollary \ref{corollary2}] 
		We recall that the triple $(1,\lambda,\mu)$ is in the Siegel (resp. Poincaré) domain if $0\in\mathbb C$ belong (resp. does not belong) to the the convex hull of $1$, $\lambda$, and $\mu$.
		
		If $\lambda\in\mathbb C\setminus\mathbb R$, up to a symmetry, we can assume that $\I(\lambda)>0$. Then, the triple $(1,\lambda,\mu)$ is in the Siegel domain if and only if the third eigenvalue $\mu$ lies in the closed region $\mathcal S:=\{z\in\mathbb C;\pi\leq\arg z\leq\pi+\arg\lambda\}$ (see figure \ref{Siegeldomain}).
		\begin{figure}[ht]
			\centering
			\includegraphics[scale=.65]{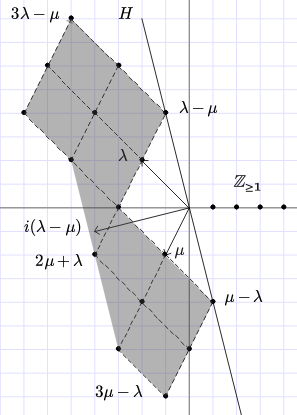}
			\caption{Siegel domain}
			\label{Siegeldomain}
		\end{figure}
		
		Assume that $\mu$ lies in $\mathcal S$. Then, two cases can occur:
		
		\begin{enumerate}[label=\roman*.]
			\item \label{corolarioitem1}$\lambda$ lies in $\mathbb C\setminus \mathbb R$, and consequently the pair $(\lambda,\mu)$ satisfies the no negative resonance condition.
			We need to prove that the discrete positive cone
			$$\mathcal C = \{p_1\lambda +p_2\mu; p_1,p_2\in\mathbb Z_{\geq-1}, p_1 +p_2 \geq 0\}$$
			contains no element $n\in\mathbb Z_{\geq1}$. We consider initially the case where $\mu = 0$. Then $\mathcal C =\{p_1\lambda; p_1\in\mathbb Z_{\geq-1}\}$ and the equality $\I(p_1\lambda) =
			\I(n)$ implies that $p_1 =0$. Consequently, we have that $0=\R(p_1\lambda)=n$. Hence $\mathcal C\cap\mathbb Z_{\geq1}=\varnothing$. 
			
			Now, assume that $\mu$ is nonzero. Since, $ \mu\in\mathcal S$, the imaginary part of $\mu$ is negative. Remark that the euclidean inner product of the vectors in $\mathbb R^2$ associated to two complex numbers $z, w$ is given by $\R(z \overline{w})$, and $iz$ and $-iz$ are orthogonal to $z$. Hence, the set $\mathcal C$ lies in the half-plane $$H=\{z;\R(z\overline{i(\lambda-\mu)})\geq 0\}$$
			\begin{figure}[ht]
				\centering
				\includegraphics[scale=.8]{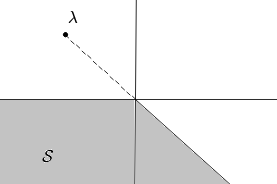}
				\caption{The discrete positive cone $\mathcal C$}
			\end{figure}
			
			Now, for an arbitrary positive integer number $n$, $$\R(n\overline{i(\lambda-\mu}))=n\R(\overline{i(\lambda-\mu)})=-n\I(\lambda-\mu)<0$$
			which shows that $n\notin H$.
			\item$\lambda$ lies in $\mathbb R$, then the region $\mathcal S$ reduces to either to the half nonpositive real line (if $\lambda>0$) or to the whole complex plane (if $\lambda \leq 0$). We consider these two cases separately. But firstly, remark that the no transverse negative resonance property is equivalent to the conjunction of the following three conditions,
			\begin{center}
				\begin{tabular}{c c r}
					$\forall p_1,p_2\in\mathbb Z_{\geq0},\forall q\in\mathbb Z_{\geq1}$;&$ p_1\lambda+p_2\mu\ne q$&$(\star)$\\
					$\forall p,q\in\mathbb Z_{\geq1};$& $p\lambda\ne\mu+ q$&$(\star\star)$\\
					$\forall p,q\in\mathbb Z_{\geq1};$& $p\mu\ne\lambda+ q$&$(\star\star\star)$
				\end{tabular}  
			\end{center}
			
			\begin{enumerate}
				\item\label{item2.1} $\mu\leq0<\lambda$. Here, $p\mu\leq0$ and $\lambda+q>0$, then the condition $(\star\star\star)$ always holds. The negation of the condition $(\star\star)$ corresponds to the equation (\ref{corolario2.0}), and the negation of the condition $(\star)$ is equivalent, up to a permutation of coordinates, to $p_1\geq1$ and $\lambda=q/p_1-p_2/p_1\mu\in\mathbb Q_{>0}-\mu\mathbb Q_{\geq0}$.
				
				\item $\lambda\leq 0$. The case where $\mu\ne\mathbb R$ is treated in (\ref{corolarioitem1}). By a changing coordinates, the case where $\mu>0$ is the item (\ref{item2.1}). Hence, up to a permutation of coordinates, just $\mu\leq\lambda\leq 0$ remains to be considered. In this case, the conditions $(\star)$ and $(\star\star\star)$ always hold. The negation of the condition $(\star\star)$ implies that $\mu<\lambda\leq 0$ and that (\ref{corolario2.0}) holds.
			\end{enumerate}
		\end{enumerate}
	\end{proof}
	\bibliographystyle{alpha}
	\bibliography{bibliography.bib} 

\begin{thebibliography}{MM80}

\bibitem[CN77]{camacho1977introduccao}
C{\'e}sar Camacho and Alcides~Lins Neto.
\newblock {\em Introdu{\c{c}}{\~a}o {\`a} Teoria das Folhea{\c{c}}{\~o}es}.
\newblock IMPA, Rio de Janeiro, 1977.

\bibitem[CN85]{camacho2013geometric}
C{\'e}sar Camacho and Alcides~Lins Neto.
\newblock {\em Geometric theory of foliations}.
\newblock BIRKHAUSER, 1985.

\bibitem[Dia19]{diaw2019geometrie}
Adjaratou~Arame Diaw.
\newblock {\em G{\'e}om{\'e}trie de certains tissus holomorphes singuliers en
  dimension 2}.
\newblock PhD thesis, Rennes 1, 2019.

\bibitem[DL20]{diaw2020pairs}
Adjaratou~Arame Diaw and Frank Loray.
\newblock Pairs of foliations and mattei-moussu's theorem.
\newblock {\em arXiv preprint arXiv:2008.00719}, 2020.

\bibitem[EI84]{elizarov1984remarks}
PM~Elizarov and Yu~S Il'yashenko.
\newblock Remarks on the orbital analytic classification of germs of vector
  fields.
\newblock {\em Mathematics of the USSR-Sbornik}, 49(1):111, 1984.

\bibitem[IY08]{ilyashenko2008lectures}
Yulij Ilyashenko and Sergei Yakovenko.
\newblock {\em Lectures on analytic differential equations}, volume~86.
\newblock American Mathematical Soc., 2008.

\bibitem[Mar81]{SB_1980-1981__23__55_0}
Jean Martinet.
\newblock Normalisation des champs de vecteurs holomorphes.
\newblock In {\em S\'eminaire Bourbaki : vol. 1980/81, expos\'es 561-578},
  number~23 in S\'eminaire Bourbaki, pages 55--70. Springer-Verlag, 1981.
\newblock talk:564.

\bibitem[MM80]{ASENS_1980_4_13_4_469_0}
Jean-François Mattei and Robert Moussu.
\newblock Holonomie et int\'egrales premi\`eres.
\newblock {\em Annales scientifiques de l'\'Ecole Normale Sup\'erieure}, 4e
  s{\'e}rie, 13(4):469--523, 1980.

\bibitem[MR82]{martinet1982problemes}
Jean Martinet and Jean-Pierre Ramis.
\newblock Problemes de modules pour des {\'e}quations diff{\'e}rentielles non
  lin{\'e}aires du premier ordre.
\newblock {\em Publications Math{\'e}matiques de l'Institut des Hautes
  {\'E}tudes Scientifiques}, 55(1):63--164, 1982.

\bibitem[Rei06]{reis2006equivalence}
Helena Reis.
\newblock Equivalence and semi-completude of foliations.
\newblock {\em Nonlinear Analysis: Theory, Methods \& Applications},
  64(8):1654--1665, 2006.

\bibitem[Ser09]{serre2009lie}
Jean-Pierre Serre.
\newblock {\em Lie algebras and Lie groups: 1964 lectures given at Harvard
  University}.
\newblock Springer, 2009.

\end{thebibliography}
\end{document}